\RequirePackage[english]{babel}

\documentclass[a4paper,10pt]{amsart}

\usepackage{amssymb, amsthm, amsmath, pifont} %f{\"u}r Symbole der reellen Zahlen
\usepackage[usenames]{color}
\usepackage[all]{xy}
\usepackage{amsfonts}
\usepackage{graphicx}
\usepackage{verbatim}
\usepackage{pstricks}
\usepackage[latin1]{inputenc}
\usepackage{mathrsfs}
\usepackage{pst-node}
\usepackage[noadjust]{cite}
\newcommand{\rien}[1]{}

\newcommand{\AVF}{ \operatorname{{\rm AVF}}}

\newcommand{\LieA}{ \operatorname{{\rm Lie}_{alg}}}

\newcommand{ \LieAO}{ \operatorname{{\rm Lie}_{alg}^\omega}}

\newcommand{\IVF}{ \operatorname{{\rm IVF}}}

\def\N{\mathbb{N}}
\def\Q{\mathbb{Q}}

\def\Z{\mathbb{Z}}
\def\C{\mathbb{C}}
\newtheorem{theorem}{Theorem}[section]
\newtheorem*{theorem*}{Theorem}
\newtheorem{corollary}[theorem]{Corollary}
\newtheorem*{corollary*}{Corollary}
\newtheorem{lemma}[theorem]{Lemma}
\newtheorem{proposition}[theorem]{Proposition}
\theoremstyle{definition}

\theoremstyle{definition}
\newtheorem{remark}[theorem]{Remark}
\theoremstyle{definition}
\newtheorem{definition}[theorem]{Definition}
\title[Lie algebra generated by LNDs on Danielewski surfaces]{Lie algebra generated by locally nilpotent derivations on Danielewski surfaces}
\author{Frank Kutzschebauch and Matthias Leuenberger}
\address{Institute of Mathematics, University of Bern, Sidlerstrasse 5, CH-3012 Bern, Switzerland}
\email{frank.kutzschebauch@math.unibe.ch}
\email{matthias.leuenberger@math.unibe.ch}

\begin{document}
\thanks{Both autors were partially supported by Schweizerischer Nationalfond Grant 200021-140235/1}
\subjclass[2010]{32M17, (32M05, 14R10, 14R20)}
\begin{abstract}
We give a full description of the Lie algebra generated by locally nilpotent derivations (short LNDs) on smooth Danielewski surfaces $D_p$ given by $xy=p(z)$.
In case $\deg(p)\geq 3$ it turns out to be not the whole Lie algebra $\mathrm{VF}_{alg}^\omega(D_p)$ of volume preserving algebraic vector fields, thus
answering a question posed
by \textsc{Lind} and the first author. Also we show algebraic volume density property (short AVDP) for a certain homology plane, a homogeneous space of the form
$SL_2 (\C) /N$, where $N$ is the normalizer of the maximal torus and another related example. 
\end{abstract}

\maketitle

\section{Introduction}
In  this paper we  study (using algebraic methods) the holomorphic automorphism group $\mathrm{Aut}_{hol}(D_p)$ of \textit{Danielewski surfaces} of the form
$D_p=\lbrace xy=p(z)\rbrace$. These surfaces are an object of intensive studies in affine algebraic geometry, see e.g. \cite{daigle1},\cite{daigle2},
 \cite{daniel}, \cite{dub4}, \cite{dub3}, \cite{dub2}, \cite{dub1}, \cite{freud}, \cite{makar1},\cite{makar2} and \cite{makar3},

The study of these surfaces from the complex analytic point of view started in the paper
of \textsc{Kaliman} and \textsc{Kutzschebauch } \cite{kkdensity},  where they proved the so called density property, or for short DP. This is a remarkable
property, discovered in 1990's  by \textsc{Anders\'en} and \textsc{Lempert}  \cite{A}, \cite{AL} for Euclidean spaces, that to a great extend 
compensates for the lack of partition of unity for holomorphic automorphisms. The terminology was
was introduced later by Varolin \cite{varo1}:  A Stein manifold $X$ has DP if the Lie algebra generated by completely integrable holomorphic
vector fields is dense (in the compact-open topology) in the space of all holomorphic vector fields on $X$. In the presence of DP one can 
construct global holomorphic automorphisms of $X$ with prescribed local properties. More precisely, any local phase flow on a Runge domain
in $X$ can be approximated by global automorphisms. Needless to say that this lead to remarkable consequences (see surveys \cite{state},
\cite{Rosay}). 

If $X$ is equipped with a holomorphic volume form $\omega$ (i.e. $\omega$ is a nowhere vanishing top holomorphic
differential form) then one can ask whether a similar approximation holds for automorphisms  and phase flows preserving $\omega$, so called \textit{volume
preserving automorphisms}.
Under a mild additional assumption the answer is yes in the presence of the volume density property (VDP) which means that the Lie algebra
generated by completely integrable holomorphic vector fields of $\omega$-divergence zero is dense in the space of all holomorphic vector fields of
$\omega$-divergence zero. 
 Danielewski surfaces carry a unique  nondegenerate algebraic 2-form $\omega$  and we will concentrate on the group $\mathrm{Aut}_{hol}^\omega(D_p)$ of volume
preserving holomorphic automorphisms.

The following definitions are due to \textsc{Varolin} and \textsc{Kaliman}, \textsc{Kutzschebauch}

\begin{definition}
We say that $X$ has the algebraic density property (ADP) is the Lie algebra $\LieA (X)$ generated by the set $\IVF (X)$ of completely integrable
algebraic vector fields coincide with the space $\AVF (X)$ of all algebraic vector fields on $X$. Similarly in the presence of $\omega$
we can speak about the algebraic volume density property (AVDP) that means the equality $\LieAO (X)=\AVF_\omega (X)$ for analogous 
objects (that is, all participating vector fields have $\omega$-divergence zero; say $\LieAO (X)$ is generated by $\IVF_\omega (X)$).
\end{definition}

It is worth mentioning that ADP and AVDP imply DP and VDP respectively (where the second implication is not that obvious) and
in particular all remarkable consequences for complex analysis on $X$.

The study of holomorphic automorphisms of Danielewski surfaces was continued  by \textsc{Lind},
 and the first author in  \cite{lind} where shear and overshear automorphisms were introduced, generalizing this notion introduced by \textsc{Rudin} and
\textsc{Rosay} from Euclidean spaces to Danielewski surfaces. 
Shears are volume preserving automorphisms whereas overshears are not. Note that the  algebraic shear vector fields are (up to coordinate change) exactly the
LNDs (see theorem \ref{shearlnd}). Generalizing the results of \textsc{Anders\'en} and \textsc{Lempert} it was proved in 
\cite{lind} that \smallskip

\textit{On a Danielewski surface the group generated by shears and overshears is dense in the path connected component of the group $\mathrm{Aut}_{hol}(D_p)$ of
holomorphic automorphisms with respect to the compact-open topology.}\

\smallskip\noindent
From the proof of DP in \cite{kkdensity} it follows that the group generated by shears, overshears and hyperbolic automorphisms is dense in
$\mathrm{Aut}_{hol}(D_p)$. The point in the above result was not to use hyperbolic automorphisms. The corresponding generalization of 
\textsc{Anders\'ens} and \textsc{Lemperts} result  in the volume preserving case, namely the question whether the group generated by shears is dense in the
group $\mathrm{Aut}_{hol}^\omega (D_p)$ of volume preserving holomorphic automorphisms with respect to the compact-open topology, remained an unsolved question 
(see \cite{lind} problem 5.1). 

In the present paper we solve the "infinitesimal
version" of this question to the negative. Namely we prove that  the algebraic shear vector fields {\bf do not generate} the Lie algebra
$\mathrm{VF}_{alg}^\omega(D_p)$ of algebraic volume preserving vector fields if the degree of the defining polynomial $p$ is at least $3$. More precisely we
prove the following statement:

\begin{corollary*}[see Corollary \ref{dense}] For $p\in\C[z]$ with degree $n \geqslant 3$ the Lie algebra generated by holomorphic shear fields is not dense
in the Lie
algebra of holomorphic volume preserving vector fields.
\end{corollary*}

If the
degree is $2$ or $1$ we  prove that  the algebraic shear vector fields {\bf do generate} the Lie algebra $\mathrm{VF}_{alg}^\omega(D_p)$ of algebraic volume
preserving vector fields. If the degree is $1$, the Danielewski surface is biholomorphic to $\C^2$ and we recover  exactly the \textsc{Anders\'en-Lempert}
result. Our main result is

\begin{theorem*}[see Theorem \ref{final}]
 A volume preserving vector field $\Theta$ on the Danielewski surface $D_p$ is a Lie combination of LNDs if and only if its corresponding function with
$i_\Theta\omega = df$ is of the form (modulo constant)
\[ f(x,y,z) = \sum_{\substack{i=1\\j=0}} ^{k}a_{ij} x^i z^j + \sum_{\substack{i=1\\j=0}}^{l}b_{ij} y^i z^j + (pq)'(z)
\] 
for a polynomial $q\in\C[z]$.
\end{theorem*}

 In the "positive" cases of
degree 1 and 2 the proof of the main theorem of Anders\'en-Lempert theory implies the density of the  group generated by shears  in the (path connected
component of the) group $\mathrm{Aut}_{hol}^\omega (D_p)$ of volume preserving holomorphic automorphisms, whereas in the "negative" cases degree $\ge 3$ we
cannot conclude that the the group generated by shears is not dense in the group $\mathrm{Aut}_{hol}^\omega (D_p)$ of volume preserving holomorphic
automorphisms. Here we are lacking a quantity attached to an automorphism which is zero for all shear automorphisms but nonzero for the hyperbolic automorphisms
 $H_f$ whose function $f$ is not the second derivative of a function divisible by the defining polynomial $p$.

 The results of our paper are also interesting in connection with the following open problem formulated in \cite{flexible}:\
 
 \textit{ Does a flexible affine algebraic manifold equipped with an algebraic volume form have the algebraic volume density property? }\

Remember that an affine algebraic manifold is called flexible if the LNDs on it generate
the tangent space at every point. By Proposition \ref{flex} this is  true  for $D_p$.

Even thought $D_p$ has the volume density property the Lie algebra generated by LNDs in not the Lie algebra $\mathrm{VF}_{alg}^\omega(D_p)$. The additional
hyperbolic fields (algebraic $\C^*$-actions) are needed to get all of $\mathrm{VF}_{alg}^\omega(D_p)$. Thus we do not have a
counterexample to the above problem, but near to a counterexample:  We have an example where the LNDs span the
tangent space at each point and at the same time do not generate the Lie algebra of volume preserving algebraic vector fields.

The paper is organized as follows. In section 2 we recall some known facts for Danielewski surfaces and give  certain proofs in order to make the paper self
contained. We believe that some of these proofs are new.

 In section 3 we explain how  volume preserving vector fields can be related to functions on the Danielewski surface and how this relation works with respect to
Lie bracket. This is a new method, which is afterwards used to prove our main result, the characterization of the Lie algebra generated by LNDs on Danielewski
surfaces. 

On the way  we use our method based on the duality between volume preserving vector fields and functions to prove (version of)  the algebraic volume density
property for $D= Sl_2 (\C) / N$, where
$N$ is the normalizer of the maximal torus $N\cong \C^*\rtimes \Z_2$. The importance of this lies in the fact that the methods (compatible pairs of globally
integrable fields)  for proving AVDP recently  developed by \textsc{Kaliman} and the first author do not work for this particular homogeneous space as explained
in \cite{KaKuvolume}. Also we prove AVDP for $(D \times \C^*) / \Z_2$ where $\Z_2$ acts diagonally. This is a good exercise, since the proof given in 
\cite{KaKuvolume} is using very abstract methods. Comparing our calculations to that proof 
let one feel  the strength of the  method of semi-compatible vector fields  developed in \cite{KaKuvolume}.

\section{Danielewski Surface}
 Let $p \in \C[z]$ be a polynomial with simple zeros. The variety given by $D_p = \lbrace(x,y,z) \in \C^3 \mid xy=p(z)\rbrace$ is called \textit{Danielewski
surface}. Since $p$ has only simple zeros $D_p$ is the preimage of a regular value and hence a complex manifold. Often it is useful to work in one of the two
charts $\C^* \times \C \rightarrow D_p: (x,z)\mapsto (x,\frac{p(z)}{x},z)$ and $(y,z) \mapsto (\frac{p(z)}{y},y,z)$, which cover all points of $D_p$ with
$x\neq0$ respective $y\neq 0$. An important fact is that every regular function $f\in\C[D_p]$ can be written uniquely as
\[ f(x,y,z) = \sum_{\substack{i=1\\j=0}} ^{k}a_{ij} x^i z^j + \sum_{\substack{i=1\\j=0}}^{l}b_{ij} y^i z^j + \sum_{i=0}^{m}c_{i} z^i \eqno{(1)}
\]
by substituting $xy = p(z)$ successively. As proven in \cite{kkvoldensity} there is an algebraic volume form $\omega$ on
$D_p$, which is unique up to a constant. In the
local charts from before it is given by $\omega=\frac{dx}{x}\wedge dz$ and $\omega=-\frac{dy}{y}\wedge{dz}$, respectively. Here comes the first well-known fact.

\begin{proposition} \label{sconn}
 The Danielewski surfaces $D_p$ are simply connected and we have $H^2(D_p,\C) \cong \C^{\mathrm{deg}(p)-1}$.
\end{proposition}
\begin{proof}
 It is possible to construct a strong deformation retraction onto a bouquet of $(\mathrm{deg}(p)-1)$ 2-spheres connecting the zeros of $p$. First choose a
smooth curve $\gamma:
[0,1] \rightarrow \C_z \subset D_p$ in the $z$-plane connecting the zeros of $p$ and then retract $D_p$ onto the spheres around the segments of the path between
the zeros. Let $\rho_t: [0,1]\times \C_z \rightarrow \C_z$ be a strong deformation retraction onto $\gamma$. We use this retraction to define the strong
deformation retraction 
\[{R}_t: D_p \rightarrow \lbrace (x,y,z)\in D_p: \ z \in \gamma \rbrace: (x,y,z) \mapsto \left(\frac{p(\rho_t(z))}{p(z)}x,y,\rho_t(z)\right).
\]
Additionally we define a strong deformation retraction $H_t$ from $\lbrace (x,y,z)\in D_p:  \ z\in\gamma \rbrace$ onto a bouquet of 2-spheres.
\[
 H_t(x,y,z):=\left(\frac{p(z)}{t|p(z)|^{1/2} \frac{y}{|y|} + (1-t)y}, t|p(z)|^{1/2} \frac{y}{|y|} + (1-t)y, z\right)
\]
for $p(z)\neq 0$ and $|y|\geq |p(z)|^{1/2}$ and
\[
H_t(x,y,z):=\left(t|p(z)|^{1/2} \frac{x}{|x|} + (1-t)x,\frac{p(z)}{t|p(z)|^{1/2} \frac{x}{|x|} + (1-t)x}, z\right)
\]
for $p(z)\neq 0$ and $|x|\geq |p(z)|^{1/2}$. When $p(z) = 0$ then either $x = 0$ or $y=0$ (or both). In this case choose
\[
 H_t(x,y,z):=\left(0,(1-t)y,z\right) \quad \mathrm{or} \quad H_t(x,y,z):=\left((1-t)x,0,z\right).
\]
The composition of $R_t$ and $H_t$ is the desired strong deformation retraction from $D_p$ to the bouquet of 2-spheres,
therefore $D_p$ is simply connected and has $H^2(D_p,\C) \cong
\C^{\mathrm{deg}(p)-1}$.
\end{proof}

\subsection{Vector fields on the Danielewski surface}
Let us begin with two equivalent definitions of locally nilpotent derivations:
\begin{definition}
 A globally integrable vector field $\Theta$ is a \textit{locally nilpotent derivation (LND)} iff its flow $\psi^t$ is an algebraic $\C^+$-action, i.e. $t
\mapsto \psi^t$ is an algebraic map. Equivalently a vector field $\Theta$ is a LND whenever for all $f\in \C[D_p]$ there is an integer $N$ such that
$\Theta^N(f)=\Theta \circ \ldots \circ \Theta(f) = 0$. For the equivalence of these definitions see \cite{LND} p.31. The subgroup of
$\mathrm{Aut}_{\mathrm{alg}}(D_p)$
generated by flows from LNDs is called the \textit{special automorphism group} $\mathrm{SAut}_{\mathrm{alg}}(D_p)$.
\end{definition}
\begin{definition} \label{shear}
 The algebraic vector fields of the Danielewski surface $D_p$ \newline
\[
 SF_i^x:=p'(z)x^i\frac{\partial}{\partial y} + x^{i+1}\frac{\partial}{\partial z}, \]
\[
SF_i^y:=p'(z)y^i\frac{\partial}{\partial x} + y^{i+1}\frac{\partial}{\partial z}
\]
are called \textit{shear fields} for all $i\in\N_0$ and the vector fields 
\[ HF_f:=f(z)\Big(x\frac{\partial}{\partial x} -y\frac{\partial}{\partial y}\Big)\]
are called \textit{hyperbolic fields} for all $f\in\C[z]$.

 The vector fields above are globally integrable and volume preserving, their flows are:
\[\phi_1^t: (x,y,z)\mapsto (x,\frac{p(z+tx^{i+1})}{x},z+tx^{i+1}), \]
\[\phi_2^t: (x,y,z)\mapsto (\frac{p(z+ty^{i+1})}{y},y,z+ty^{i+1}), \]
\[\phi_3^t: (x,y,z)\mapsto (e^{tf(z)}x,e^{-tf(z)}y,z). \]
Note that $\frac{p(z+tx^{i+1})}{x} = \frac{p(z) +tx^{i+1}(\ldots)}{x} = y + tx^i(\ldots)$. This shows that the shear fields are locally nilpotent derivations
and the hyperbolic fields are not. 
For $t=1$ this automorphisms are called \textit{x-(resp y-)shear automorphisms} (short: \textit{shears}) respectively \textit{hyperbolic automorphisms}. 
\end{definition}

Recall the following definition from  \cite{flexible}.
\begin{definition}
 $M$ is said to be \textit{flexible} iff the LND-vector fields span the tangent space in all points of $M$. For properties of flexible manifolds see
\cite{flexible}.
\end{definition}

\begin{proposition}[\cite{verytransitive}]\label{flex}
 The Danielewski surface is flexible.\end{proposition} 
 
 \begin{proof}
 
 The two following LND-vector fields span the tangent space in every point of $D_p$ where $p'(z) \neq 0$.
\[
 p'(z)\frac{\partial}{\partial y} + x\frac{\partial}{\partial z}, \quad p'(z) \frac{\partial}{\partial x} + y\frac{\partial}{\partial z}.
\]
For the points with $p'(z) = 0$ we add the following vector fields. Let $\alpha_k(x,y,z)=(x,\frac{p(z-kx)}{x},z-kx)$ then the with $\alpha_k$ conjugated
shear fields do the job (see the remark below)
\begin{eqnarray*}
\alpha_k^*(SF_0^y) &=& p'(z+kx) \frac{\partial}{\partial x} \\ &&+ \frac{p(z+kx)p'(z) -p'(z+kx)p(z) - kxp'(z+kx)p'(z)}{x^2}\frac{\partial}{\partial y} \\ && +
\left(-kp'(z+kx) + \frac{p(z+kx)}{x}\right)\frac{\partial}{\partial z}.
\end{eqnarray*}
Assume $p'$ has $n$ zeros, then the fields $\alpha_k^*(SF_0^y)$ for $k=1, \ldots, n$ together with the two shear fields from above will span the tangent space
at any point.

\end{proof}

\begin{remark} \label{conj}
 Given a vector field $\Theta$ and a holomorphic automorphism $\phi: M \rightarrow M$ then the \textit{vector field conjugated by $\phi$} is given by $(\phi^*
\Theta)_p := ((D\phi^{-1}) \Theta)_{\phi(p)}$. The vector field $\phi^*\Theta$  is globally integrable whenever $\Theta$ is it. Its flow is $\phi \psi^t
\phi^{-1}$ where $\psi^t$ is the flow of $\Theta$. In particular a LND conjugated by an algebraic automorphism is an LND again. The interior product for a
k-form $\omega$ is $i_{(\phi^*\Theta)}\omega = \phi^*(i_\Theta(\phi^{-1*}\omega))$, in particular when $\omega$ is invariant under $\phi$ then
$i_{(\phi^*\Theta)}\omega = \phi^*(i_\Theta\omega)$.
\end{remark}

\subsection{The (Special) Automorphism Group}
The goal of this subsection is to see that the LNDs are exactly the shear fields and shear fields conjugated with shear automorphisms. This  result is not new
\cite{makar3}, in order to make the paper self contained we give a proof (which to our knowledge is new).

 We begin with the description of the algebraic automorphism group $\mathrm{Aut}_{\mathrm{alg}}(D_p)$. The following theorem is due to  \textsc{Makar-Limanov},
he stated it in
the end of the paper \cite{makar1} without proving it. 

\begin{theorem}[\cite{makar1}]\label{makar} %%%exclude symmetries of the polynomial p
 Let $\mathrm{deg}(p)\geq 3$ and let $p$ be generic in the following sense: No affine automorphism $\alpha$ of $\C$ permutes the roots of $p$. Then the group of
all algebraic automorphisms is $\mathrm{Aut}_{\mathrm{alg}}(D_p) = G_0 \rtimes (H \rtimes J)$ where $G_0 = G_x * G_y$ is the free product of the subgroups $G_x$
(resp. $G_y$)
generated by the x- (resp. y-) shear automorphisms, $H$ is the subgroup of algebraic hyperbolic automorphisms and $J$ is the subgroup consisting of the identity
and $I(x,y,z)=(y,x,z)$ is the involution.\ 
 
 In the non generic case denote by $\Gamma$ the group of affine automorphisms of $\C$ permuting the roots of $p$, i.e. $p \circ \gamma = a_0 \ p(z)$, where
$a_0$ is a root of unity (depending on $\gamma$).
 $\Gamma$ induces a group of automorphisms of $D_p$, which we denote by $\tilde \Gamma$.
 In this case we denote by $J$ the group generated by $\tilde \Gamma$ and $I$,  then we have again $\mathrm{Aut}_{\mathrm{alg}}(D_p) = G_0 \rtimes (H \rtimes J)
$ with $G_0$
and $H$ as above.
\end{theorem}

We will give a proof using the following main theorem in \cite{makar1} .

\begin{theorem}[\cite{makar1}]  %%%exclude symmetries of the polynomial p
Let $\mathrm{deg}(p)\geq 3$ and let $p$ be generic as above, then 
 the group of algebraic automorphisms of $D_p$ is generated by the following automorphisms: 
\begin{list}{}{}
\item x-shears: $\Delta_f(x,y,z)=\left(x,\frac{p(z+xf(x))}{x},z+xf(x)\right)$ for $f\in\C[z]$
\item Hyperbolic rotations: $H_\lambda(x,y,z)=(\lambda x, \lambda^{-1} y, z)$ for $\lambda \in \C^*$
\item Involution: $I(x,y,z)=(y,x,z)$.
\end{list}
Note that y-shears are exactly the automorphisms of the form $I\Delta_f I$. \\

In the non generic case or if $\mathrm{deg}(p) = 2$ one has to add (the finite group) $\tilde \Gamma$ of automorphisms coming from symmetries of $p$:
\begin{list}{}{}
\item  $\gamma (x, y, z) = (x, a_0 y, \gamma (z))$, where $\gamma (z) = a_0 z + b$ is such that $ p \circ \gamma (z) = a_0 p(z)$
\end{list}

\end{theorem}
\begin{proof}
 see \cite{makar1}
\end{proof}
\begin{lemma}
 For $\mathrm{deg}(p) \geq3$ a nontrivial composition of x- and y- shears will never have a z-coordinate of the form $az+b$.
\end{lemma}
\begin{proof}
 Since composition of x- (resp. y-) shears are x- (resp. y-) shears again, $G_x$ and $G_y$ are subgroups and we can assume that the composition is written in a
reduced way (i.e. alternating x- and y- shears). For instance take an element $\Delta_{f_n}^x \Delta_{f_{n-1}}^y \cdots \Delta_{f_2}^y\Delta_{f_1}^x$ (the
letter
$\lbrace x,y \rbrace$ denotes whether it is a x- or a y-shear). Denote the image of $(x,y,z)=(x_0,y_0,z_0)$ after the first $i$ shears by $(x_i,y_i,z_i)$ e.g.
for $i$ odd we get
\[x_i=x_{i-1}, \ y_i=\frac{p(z_{i-1} + x_{i-1}f_i(x_{i-1}))}{x_{i-1}} \ \ \mathrm{and} \ z_i=z_{i-1} + x_{i-1}f_i(x_{i-1}).\] Since $y=\frac{p(z)}{x}$ we can
see the elements $x_i,y_i, z_i$ as unique elements in $\C[x,x^{-1},z]$ and therefore it makes sense to speak of the x-degree of such an element. The goal
is to see that $z_i$ has a strictly positive x-degree for $i>0$ and is therefore not of the form $az+b$. After the first shear $z_1=z + xf_1(x)$ is obviously of
positive x-degree, more precisely it has degree $\mathrm{deg}(f_1)+1$. Composing inductively with the proceeding shear automorphism a term $x_{i}f_{i+1}(x_{i})$
or
$y_{i}f_{i+1}(y_i)$ will be added. If we can see that the x-degree of a such term is always bigger than all previous ones then the claim is proven. Indeed the
x-degree of $y_{i}f_{i+1}(y_{i})$ is $\mathrm{deg}(y_i)(\mathrm{deg}(f_{i+1})+1)\geq\mathrm{deg}(y_i)=\mathrm{deg}(p)\mathrm{deg}(z_{i-1} + x_{i-1}f_i(x_{i-1}))
- \mathrm{deg}(x_{i-1})$ which is by induction $\mathrm{deg}(p)\mathrm{deg}(x_{i-1})(\mathrm{deg}(f_i) +1) -
\mathrm{deg}(x_{i-1})=\mathrm{deg}(x_{i-1})(\mathrm{deg}(p)(\mathrm{deg}(f_i)+1) -1)>\mathrm{deg}(x_{i-1})(\mathrm{deg}(f_i)+1) =
\mathrm{deg}(x_{i-1}f_i(x_{i-1}))$. The inequality from the second last step follows from the fact that $\mathrm{deg}(p)\geq3$. The same calculation holds for
$x_{i}f_{i+1}(x_{i})$. And if our arbitrary elements starts with a y-shear then of course the same calculation holds when we exchange $x$ and $y$. 
\end{proof}

\begin{proof} [Proof of theorem \ref{makar}] 
To see that $\mathrm{Aut}_{\mathrm{alg}}(D_p) = G_0 \rtimes (H \rtimes J)$  in the generic case it is necessary to verify several things. First we see that
$\mathrm{Aut}_{\mathrm{alg}}(D_p)
= G_0 \rtimes H_0$ where $G_0$ is the group generated by automorphisms of the form $\Delta_f$ and $I\Delta_f I$ and $H_0$ is generated by automorphisms
$H_\lambda$ and $I$. $G_0$ is indeed normal since $I\Delta_f I$ and $II\Delta_f II = \Delta_f \in G_0$ and $H_\lambda^{-1}\Delta_f H_\lambda = \Delta_{\lambda
f(\lambda \cdot)} \in G_0$. Since $IH_\lambda = H_\lambda^{-1} I$ we have $h^{-1}gh\in G_0$ for all elements $h \in H_0$ and $g\in G_0$. By the theorem above it
is clear that $G_0$ and $H_0$ generate $\mathrm{Aut}_{\mathrm{alg}}(D_p)$ so the last thing to check is that the intersection is trivial. We observe that all
elements of $H_0$
fix the z-coordinate but no nontrivial element from $G_0$ does by the previous lemma. Take a look at the surjective homomorphism $G_x * G_y \rightarrow G_0$
sending a word to its interpretation in the group, to see that it is injective it is sufficient to see that the identity map can't be written as a nontrivial
composition of shear automorphisms, but this is clear since a nontrivial composition of shears never fixes the third component. To finish the proof we have to
see that $H_0$, the subgroup generated by hyperbolic rotations and the involution is $H \rtimes J$, but this is clear since $IH_\lambda I = H_\lambda^{-1}$ and
therefore the subgroup $H$ generated by hyperbolic rotations is normal and $I$ is orientation reversing and therefore not part of the hyperbolic rotations. The
statement in the non generic case is easy to see as well.
\end{proof}
 Here are some consequences of the theorem, remember that all of them hold just for $\mathrm{deg}(p)\geq 3$.
\begin{remark}
 In the generic case the group of algebraic volume preserving automorphisms is therefore $\mathrm{Aut}_{\mathrm{alg}}^\omega(D_p) = G_0 \rtimes H$, since shears
and hyperbolic
automorphism are volume preserving and the involution is volume reversing. The (non trivial) elements of $\tilde \Gamma$ from the
non generic case multiply the volume form by a (non zero) root of unity, so the group can be bigger since
it is possible to get an order two volume preserving automorphism of the form $I\circ \gamma$ with $\gamma \in \tilde\Gamma$. In this case the group of volume
algebraic volume preserving automorphisms is $G_0 \rtimes(H\rtimes\Z_2)$.
\end{remark}
\begin{proposition} \label{propspecial}
 The group of special automorphisms $\mathrm{SAut}_{\mathrm{alg}}(D_p)$ (i.e. the group generated all algebraic $\C^+$-actions) is the
group $G_0\cong G_x * G_y$ generated by the shear automorphisms.
\end{proposition}
\begin{proof}
 Take any algebraic one parameter subgroup $\psi: \C \rightarrow \mathrm{Aut}_{\mathrm{alg}}(D_p)$. Since we have the projection homomorphism
$\mathrm{Aut}_{\mathrm{alg}}(D_p)=G_0 \rtimes (H \rtimes J)\rightarrow H\rtimes J$  we get an induced algebraic one parameter subgroup on $H \rtimes J$ and
hence on its
connected component $H$ the subgroup of hyperbolic rotations, but this subgroup has to be trivial since one parameter subgroups in $H$ can never be algebraic
$\C^+$ action. Hence $\psi$ has its image in the shear automorphisms.
\end{proof}
\begin{lemma} \label{free}
 A smooth one parameter subgroup $\psi: \C \rightarrow G_x * G_y$ is conjugated to a one parameter subgroup $\psi^t$ either in $G_x$ or in $G_y$.
\end{lemma}
In order to prove this lemma we need some facts about free groups. Recall that for two groups $G$ and $H$ any element $g$ in $G * H$ has a unique reduced form
with length denoted by $l(g)$.
\begin{theorem}
A subgroup $K$ of $G*H$ is conjugated to a subgroup in either $G$ or $H$ if and only if $\mathrm{sup}\left(l(k); k\in K \right) < \infty$.
\end{theorem}
\begin{proof}
See \cite{trees} theorem 8 p.36.
\end{proof}
The following lemma is well known, see e.g. \cite{KrKu}, for making the paper more self contained we give the proof.
\begin{lemma}\label{freelemma}\begin{enumerate}
  \item Every element in $G*H$ is conjugated either to an element in either $G$ or $H$ or to an element of even length $>0$.
\item Two commuting elements of $G*H$ with length $>0$ have either both even or both odd length.
 \end{enumerate}
\end{lemma}
\begin{proof}
 (1) Whenever an element has odd length its first and last letter belongs to the same group then after conjugating with the inverse of one of those letters it
is either of even length or the length descends by 2 and we can proceed by induction. (2) Take an element $a$ with even length $n$ and an element $b$ with odd
length $m$, then either $l(ab)=m + n$ and $l(ba)< m+n$ or $l(ab)<m + n$ and $l(ba)= m+n$ and hence they cannot commute. 
\end{proof}
\begin{proof}[Proof of lemma \ref{free}]
 First we show that for all $z\in \C$ the element $\psi(z)$ is conjugated to a shear automorphism (i.e. is conjugated to an element of either $G_x$ or $G_y$).
Assume that
this is not the case, then $a\psi(z)a^{-1}$ were of even length for some $a\in G_x * G_y$. Since $a\psi(z)a^{-1}$ and $a\psi(\frac{z}{n})a^{-1}$ commute
$a\psi(\frac{z}{n})a^{-1}$ is also of even length and
therefore $l(a\psi(z)a^{-1})=l((a\psi(\frac{z}{n})a^{-1})^n)>n$ for all $n$, which is of course a contradiction. Therefore with lemma $\ref{freelemma}$ we have
for each $z$ an
element $g_z$ such that $g_z^{-1} \psi(z) g_z$ is a shear automorphism. Now take an element $m + n\sqrt{2} +i(p + q\sqrt{2}) \in \Q[\sqrt{2},i]$:
\begin{eqnarray*}
 \psi(m + n\sqrt{2} +i(p + q\sqrt{2}))&=& \psi(1)^m\psi(\sqrt{2})^n\psi(i)^p\psi(i\sqrt{2})^q\\
&=&g_1(g_1^{-1}\psi(1)g_1)^m g_1^{-1}g_{\sqrt{2}}(g_{\sqrt{2}}^{-1}\psi(\sqrt{2})g_{\sqrt{2}})^m g_{\sqrt{2}}^{-1} \\
&& \cdot g_i(g_i^{-1}\psi(i)g_i)^m g_i^{-1}g_{i\sqrt{2}}(g_{i\sqrt{2}}^{-1}\psi(i\sqrt{2})g_{i\sqrt{2}})^m g_{i\sqrt{2}}^{-1}.
\end{eqnarray*}
 Therefore the length of elements in $\psi(\Q[\sqrt{2},i])$ is bounded by $2(l(g_1)+l(g_i)+l(g_{\sqrt{2}})+l(g_{i\sqrt{2}})) +4$ and hence
$\psi(\Q[\sqrt{2},i])$ is by lemma $\ref{free}$ conjugate to a subgroup of $G_x$ or $G_y$. Now the only thing remained to show is that $G_x$ and $G_y$ are
closed in $G_x * G_y$ then we know that also \[\psi(\C) = \psi(\overline{\Q[\sqrt{2},i]})\subset\overline{\psi(\Q[\sqrt{2},i])}\] is conjugate to a subgroup of
$G_x$ or $G_y$. To see that for instance $G_x$ is closed we take any converging sequence of x-shears $\Delta_{f_n} \rightarrow \eta=(\eta_1,\eta_2,\eta_3)$. So
we know that
$(z + f_n(x))_n$ converges point-wise hence $f_n(z)$ converges, say to $f(z)$. Now clearly $\eta_1(x,y,z)=x$ and $\eta_3(x,y,z)=z+f(x)$, since $\eta$ is
algebraic $f$ is
a polynomial and therefore $\eta=\Delta_f$ is an x-shear.
\end{proof}

\begin{theorem}[\cite{makar3}] \label{shearlnd}
 The LNDs of the Danielewski surface $D_p$ for $\mathrm{deg}(p)\geq 3$ are exactly the shear fields and the shear fields conjugated by compositions of shear
automorphisms.
\end{theorem}
\begin{proof}
 An algebraic $\C^+$-action $\psi: \C \rightarrow \mathrm{SAut}_{\mathrm{alg}}(D_p)$ is by Proposition $\ref{propspecial}$ and Lemma $\ref{free}$
conjugated to a one parameter subgroup in $G_x$ or $G_y$.
\end{proof}

\section{Lie Combinations of Shear Fields}
In this chapter we will understand which algebraic volume preserving vector fields of the Danielewski surface can be written as a Lie combination of the shear
fields $\ref{shear}$. The main tool for the description will be the 1-forms $i_\Theta\omega$ for volume preserving vector fields. Recall that the
\textit{interior product} $i_\Theta: \Omega^{k+1}(M) \rightarrow \Omega^{k}(M)$ is given by $i_\Theta\mu(\Theta_1, \ldots , \Theta_k) := \mu(\Theta,\Theta_1,
\ldots , \Theta_k)$. We will also use the \textit{Lie derivative} of a differential form $\mu$ with respect to a vector field $\Theta$, which is given by
$L_\Theta \mu = \frac{d}{dt}\psi^{t*}\mu\mid_{t=0}$ or the Cartan formula $L_\Theta \mu = (d\circ i_\Theta + i_\Theta\circ d)(\mu)$. The formula
$i_{[\Theta_1,\Theta_2]}\mu = L_{\Theta_1}(i_{\Theta_2}\mu) - i_{\Theta_2}(L_{\Theta_1}\mu)$ gives a link between the interior product and the Lie derivative.
Another useful formula $L_\Theta d\mu = dL_\Theta \mu$ is a direct consequence of the Cartan formula.

\subsection{The Lie algebra generated by shear fields is a proper subalgebra of $\mathrm{VF}_\mathrm{alg}^\omega(D_p)$} 
\smallskip\noindent
From now on we will use the one-one correspondence between algebraic volume preserving vector fields and polynomial functions modulo constants on $D_p$. For
every volume preserving vector field $\Theta$ holds $L_{\omega}(\Theta) = di_{\Theta}\omega+i_{\Theta}d\omega = 0$. Since $d\omega=0$ the 1-form
$i_{\Theta}\omega$ is closed and therefore exact (because $D_p$ is simply connected $\ref{sconn}$), hence when $\Theta$ is algebraic then $i_{\Theta}\omega =
df$ for some regular $f\in\C[D_p]$. This defines a bijection between algebraic volume preserving vector fields and polynomial functions modulo constants. 

This correspondence is in analogy to the correspondence between symplectic vector fields and Hamiltonian functions in symplectic geometry (on simply connected
symplectic manifolds). In other words we use the structure of Poisson algebra on the functions on the manifold. This analogy is using the facts that $\omega$
is closed and non-degenerate. If we consider higher dimensional manifolds (not surfaces) the correspondence will be between volume preserving vector fields and
$n-2$ forms, see \cite{KaKuvolume}.
The
following lemma gives the corresponding functions to the shear fields and hyperbolic vector fields.

\begin{lemma} \label{satzi}
 For $i\in\N_0$ holds:
\[ i_{SF_i^x}\omega = -\frac{dx^{i+1}}{i+1}, \quad i_{SF_i^y}\omega = \frac{dy^{i+1}}{i+1}, \quad i_{HF_{z^i}}\omega = \frac{dz^{i+1}}{i+1}.\]
\end{lemma}

\begin{proof}
\begin{eqnarray*}
 i_{SF_i^x}\omega(\Theta)&=&\omega(SF_i^x,\Theta)= \frac{1}{x}dx\wedge dz(SF_i^x,\Theta) \\&=&  \frac{1}{x}\Big(dx(SF_i^x)dz(\Theta)- dx(\Theta)dz(SF_i^x)\Big)
\\&=& \frac{1}{x}\Big(-x^{i+1}dx(\Theta)\Big) = - x^i dx(\Theta) = -\frac{dx^{i+1}}{i+1}(\Theta).\end{eqnarray*}
\begin{eqnarray*}
i_{SF_i^y}\omega(\Theta)&=&\omega(SF_i^y,\Theta)= -\frac{1}{y}dy\wedge dz(SF_i^y,\Theta) \\&=& -\frac{1}{y}\Big(dy(SF_i^y)dz(\Theta)- dy(\Theta)dz(SF_i^y)\Big)
\\&=& -\frac{1}{y}\Big(-y^{i+1}dy(\Theta)\Big) = y^i dy(\Theta) = \frac{dy^{i+1}}{i+1}(\Theta).
\end{eqnarray*}
\begin{eqnarray*}
i_{HF_{z^i}}\omega(\Theta)&=&\omega(HF_{z^i},\Theta)= \frac{1}{x}dx\wedge dz(HF_{z^i},\Theta) \\&=& \frac{1}{x}\Big(dx(HF_{z^i})dz(\Theta)-
dx(\Theta)dz(HF_{z^i})\Big) \\& = &\frac{1}{x}(z^i xdz)= z^i dz(\Theta) = \frac{dz^{i+1}}{i+1}(\Theta).
\end{eqnarray*}
\end{proof}
In general, it is not hard to see that for a given function $f$ the corresponding vector field $\Theta$ is given by
\[\Theta = \left(p'(z)f_y+ xf_z\right)\frac{\partial}{\partial x}- \left(p'(z)f_x + y f_z\right) \frac{\partial}{\partial y} + \left(yf_y -
xf_x\right)\frac{\partial}{\partial z},\]
where $f_x,f_y,f_z$ denote the partial derivatives of $f$.
We need to know how to calculate the Lie bracket on the level of functions. An easy calculation shows the following lemma.
\begin{lemma} \label{lemi}  Let $\Theta$ be a volume preserving vector field with $i_{\Theta}\omega = df$ and $\Psi$ another volume preserving vector field,
then \[i_{[\Psi,\Theta]}\omega = L_{\Psi}(i_{\Theta}\omega) - i_{\Theta}(L_{\Psi}(\omega)) = L_{\Psi}(df) = dL_{\Psi}(f).\]\end{lemma}
This lemma also allows us the compute the Lie bracket only in terms of functions (which is usually called the Poisson bracket):
\[
 \lbrace f,g \rbrace = p'(z)(f_yg_x-f_xg_y) + x(f_zg_x - f_xg_z) -y(f_zg_y-f_yg_z),
\]
however we will never use this precise description.

The previous facts allow us to reprove the fact from \cite{KaKuvolume}   that $D_p$ has the volume density property.
\begin{theorem} \label{satzchar}
 The Danielewski surface $D_p$ with the volume form $\omega$ satisfies the algebraic volume density property, in fact every algebraic volume preserving vector
field is a Lie combination of shear fields and hyperbolic fields. Precisely: Every volume preserving vector field is a linear combination of vector fields
$SF_i^x$, $SF_i^y$, $HF_f$, $[SF_i^x,HF_f]$ and $[SF_i^y,HF_f]$ for $i\in\N_0$ and polynomials $f\in\C[z]$.
\end{theorem}
\begin{proof}
 We have to find a Lie combination $A$ of shear fields and hyperbolic fields for every polynomial function $f$ on $D_p$ such that $i_{A}\omega = df$ holds. It
is sufficient to find the corresponding Lie combination for the monomials $x^i$, $y^i$, $z^i$, $x^i z^j$ and $y^i z^j$ for all $i,j > 0$, but the first three
are already covered by lemma $\ref{satzi}$. The corresponding vector fields of the last two monomials are $[SF_{i-1}^x,HF_{z^j}]$ and $[SF_{i-1}^y,HF_{z^j}]$,
indeed:
\begin{equation*}
 i_{[SF_{i-1}^x,HF_{z^j}]}\omega =dL_{SF_{i-1}^x}\left(\frac{z^{j+1}}{j+1}\right) = d x^i \frac{1}{j+1}(j+1)z^j=dx^i z^j
\end{equation*}
 A similarly calculation shows $i_{[SF_{i-1}^y,HF_{z^j}]}\omega = dy^i z^j$. 
\end{proof}

Now we have developed the method  to show AVDP for the cases mentioned in the introduction.  Let $D$ be the quotient of $SL_2 (\C)$ by the normalizer of the
maximal torus.
Consider $G = SL_2 (\C)$ as a subvariety of $\C^4_{a_1,a_2,b_1,b_2}$  given by $a_1b_2-a_2b_1=1$,
i.e. matrices
$$A =  \left[ \begin{array}{rr}
a_1& a_2  \\
b_1 & b_2 \\
\end{array}  \right]$$ are elements of $G$.
Let $T\simeq \C^*$ be the torus consisting of the diagonal elements and $N$ be the normalizer of $T$
in $SL_2$. That is, $N/T \simeq \Z_2$ where the matrix
$$A_0= \left[
\begin{array}{rr}
0 & -1  \\
1 & 0 \\
\end{array}  \right] \in N$$ generates the nontrivial coset of $N/T$.

\begin{lemma} The variety $D=G/T$ is isomorphic to the hypersurface $xy=z^2-1$ in $\C^3_{x,y,z }$ such that the $\Z_2$-action
is given by $(x ,y,  z)\to (-x,-y,-z)$.
\end{lemma}

\begin{proof} Note that the ring of $T$-invariant regular functions on $G$ is generated by $x=a_1b_1, y =a_2b_2, v=a_1b_2$, and $z=a_2b_1$
where $v=z+1$. Hence $X$ is isomorphic to the hypersurface $x y =z(z+1)$ in $\C^3_{x, y, z,}$. After a linear isomorphism of $\C^3$ we get the desired form. The
formula for
the $\Z_2$-action (induced by multiplication by $A_0$) is also a straightforward computation.
\end{proof}

\begin{definition} Let $X$ be an affine algebraic manifold equipped with an algebraic volume form $\omega$. Suppose a finite group $\Gamma$ acts freely and
algebraically
on $X$. We say that $X$ has the $\Gamma$-AVDP if the Lie algebra generated by  $\Gamma$-invariant completely integrable volume preserving algebraic vector
fields on $X$ is
equal to the Lie algebra of all  $\Gamma$-invariant  volume preserving algebraic vector fields on $X$.
\end{definition}

\begin{theorem} \label{z2}
The Danielewski  surface $D$ has  $\Z_2$-AVDP. 
\end{theorem}

\begin{proof}

We proceed as in the proof of the previous theorem, The volume form $\omega$ is $\Z_2$ anti-invariant, i.e., $\sigma^* \omega= - \omega$. Thus  using the
invariant globally integrable fields $SF_{2n}^x$,  $SF_{2n}^y$, $HF_{2n}$ $n\ge 0$ we have to produce all anti-invariant monomials $x^i$, $y^i$,  $z^i$ for odd
$i$ and $z^i x^j$, $z^i y^j$ for $i, j \ge 1$,  $i+j \ge 3$ and odd. The first three are again covered by  lemma $\ref{satzi}$ for even $ i$.

 For the other monomials we have to use the exact form of the the defining polynomial $p(z) = z^2-1$. We obtain the monomials $z^i x^j$ by induction on $i$. The
monomials $ z^i y^j$ are then obtained analogously.

Starting the induction with $i=1$ consider  $$i_{[SF_0^y, SF_{2k}^x]} \omega = d SF_0^y ( -\frac{x^{2k+1}}{2k+1} ) = d \left( (2z \frac{\partial}{\partial x} +
y \frac{\partial}{\partial z})  ( -\frac{x^{2k+1}}{2k+1} ) \right) =
- 2d z x^{2k}$$

Suppose by induction hypothesis that all monomials $z^m x^n$, $m+n$ odd  for $m\le i$ are obtained. To produce a monomial $ z^{i+1} x ^{j}$ use the Lie bracket
of $SF_0^y$
with the field corresponding to the monomial $ z^i x^{j+1}$ (which by induction hypothesis is obtained). We obtain the polynomial

$$  ( 2z \frac{\partial}{\partial x} +  y \frac{\partial}{\partial z}) (z^i x^{j+1}) = 2 z^{i+1} (j+1) x^j + i z^{i-1}  y x^{j+1} =$$ $$= 2 z^{i+1} (j+1) x^j +
i z^{i-1}  (z^2 -1)  x^{j} =
(2 j + 2 + i) z^{i+1} x^j - i z^{i-1} x^j     $$

The monomial $z^{i-1} x^j $ is already obtained by induction hypothesis, thus the induction step is completed.

We do not get constant functions, they are not needed since they correspond to the zero field.
\end{proof} 

In fact the use of Lie brackets is not necessary in  the previous theorem, one can show that linear span is enough.

\begin{remark} 
 The vector space (instead of Lie  algebra) spanned by   globally integrable $\Z_2$-invariant algebraic vector fields on $D$ is equal to all $\Z_2$ invariant
algebraic vector fields. Also the vector space spanned  by   globally integrable $\Z_2$-anti-invariant algebraic vector fields on $D$ is equal to all $\Z_2$
anti-invariant algebraic vector fields\footnote{  For the anti-invariant case one shows exactly as in the proof above  that all anti-invariant fields
are obtained as Lie brackets of one anti-invariant globally integrable field and invariant LNDs.}.
\end{remark}

This follows from the fact that in the above proof one uses Lie brackets of LNDs and maximally one other (hyperbolic) globally integrable field and the
following general fact
which holds on any affine algebraic manifold.

\begin{lemma}
If $\Theta$ is an LND and $\Psi$ a finite sum of globally integrable algebraic vector field, then the Lie bracket $[\Theta, \Psi]$ is contained in the span of
globally integrable algebraic vector fields.
In particular the vector space spanned by LNDs is equal to the Lie algebra generated by LNDs.
\end{lemma}

\begin{proof}
Let $\phi_t$ denote the flow of $\Theta$ (which is an algebraic $\C$-action). Then the set  $A=\{(\phi_t)^* (\Psi) \} $ is contained in a finite dimensional
subspace of $\AVF$ and thus its span is closed (see Lemma \ref{lemclosed}). Since global integrability is preserved when applying an automorphisms, all fields
in $A$ are in the span of
globally integrable fields. Moreover the definition $$[\Theta, \Psi] = lim_{t\to 0} \frac{(\phi_t)^* (\Psi) -\Psi}{t}$$ shows that the bracket  $[\Theta, \Psi]$
is in the closure of the span of $A$, thus in the span.
\end{proof}

Now the other example: Let $X= D \times \C^*$ equipped with the volume form $\omega_0 = \omega \times \frac{ d \theta}{ \theta }$ and $\Z_2$-action generated by
$(x, y, z, \theta) \mapsto (-x, -y, -z, -\theta)$. The next theorem states that $X$ has $\Z_2$-AVDP, the proof technique is very close the technique we have
seen above. For a vector field $\Theta$ we again look at the corresponding form $i_\Theta\omega_0$ which is in this situation an anti-invariant closed 2-form.
In order to find all those forms we need to find all anti-invariant exact 2-forms and additionally for each cohomology class one representative.

\begin{theorem}[\cite{KaKuvolume}] \label{z3}
The manifold $X$ has  $\Z_2$-AVDP. 
\end{theorem}

\begin{proof}
The volume form $\omega_0$ is anti-invariant. We wish to find all  anti-invariant closed 2-forms $\alpha$ on $X$ as $i_\chi \omega_0$ where $\chi$ is a Lie
combination of
invariant completely integrable fields on $X$. By Proposition \ref{sconn}   $H^2 (D, \C) = \C$ and it is easy to check that the volume form $\omega$ represents
the nontrivial class.
By K\"unneth formula and $H^1(D,\C)=0$ we have that $H^2 (X, \C)$ is isomorphic to $\C$ and $\omega$ (considered as a 2 -form on $X$) is a generator. Remark
that $ \omega = i_{\theta
\frac{\partial}{\partial \theta}}
\omega_0$. Thus subtracting the completely integrable
volume preserving invariant field $ \theta \frac{\partial}{\partial \theta} $ from a given field $\chi$ we can assume that the form $\alpha$ is exact. It
remains  to construct all anti-invariant $1$-forms $\beta$ in the expression $ d\beta = i_\chi \omega_0$ where $\chi$ is a Lie combination of invariant
completely integrable fields on $X$.
Of course we have to find all $1$-forms $\beta$ up to closed ones, since these correspond to the zero vector field.

Since the restrictions of the 1-forms $dx$, $dy$ and $dz$ from $\C^3$ to the tangent space of $D$ generate the cotangent space of $D$ at any point,  all
$1$-forms on $X$ can be written as 
$$\beta = \sum_{n = -N}^{N}  f_n (X) \theta^n dx  +   \sum_{n = -N}^{N}  g_n (X) \theta^n dy +  \sum_{n = -N}^{N}  h_n (X) \theta^n dz +  \sum_{n = -N}^{N}  j_n
(X) \theta^n d \theta$$

where $X = (x, y, z)$ and $ f_n, g_n, h_n, j_n$ are regular functions on $D$ which are invariant if $n$ is even and anti-invariant if $n$ is odd. Of course this
representation of a 1-form on $X$  is not unique, the relation $x dy + y dx = 2z dz$ holds, but this is irrelevant for our proof. 

We begin by  constructing all summands of the fourth sum. First consider the case of even $n$. The proof is analogous to the proof of the preceding theorem. The
monomial forms
$x^i \theta^n d \theta $, $i$ even,  you construct by inner product of the invariant completely integrable field $\theta^{n+1}  SF_{i-1}^x$ with $\omega_0$,
$y^i \theta^n d \theta$
comes from  $\theta^{n+1}  SF_{i-1}^y$  $i$ even, and  $z^i \theta^n d \theta$ comes from the invariant field $\theta^{n+1}  HF_{i-1}$  $i$ even. Now use
inductively  Lie brackets
with  the invariant field  $SF_0^y$ to obtain out of the form $x^i \theta^n d \theta$ the forms $x^{i-1} z \theta^n d \theta$, $x^{i-2} z^2 \theta^n d \theta$
and so on thus obtaining all
$1$- forms $x^k z^l \theta^n d \theta$ for $k+l$ even. The forms $y^k z^l \theta^n d \theta$, $k+l$ even, are obtained analogously. 
Now consider the case $n$ odd. Start with the monomial forms $x^i \theta^n d \theta $, $i$ odd,  you constructed by inner product of the invariant completely
integrable field $\theta^{n+1}  SF_{i-1}^x$ with $\omega_0$, all the rest goes analogously. We thus have constructed all anti-invariant $1$- forms $\sum_{n =
-N}^{N}  j_n (X) \theta^n d \theta$, except for
$j_n = constant$, but the forms $\theta^n d \theta$ are closed and therefore corresponding to the zero field. 

In order to produce the summand in the first sum we introduce the invariant globally integrable volume preserving vector field $V=x\theta(x\partial/\partial x
-y\partial/\partial y -\theta\partial/\partial\theta)$ and take the Lie bracket with the vector field corresponding to the 1-form $f(X)\theta^nd\theta$ (say
$n$ even and $f$ invariant). This produces the 1-form
\[
 L_V(f(X)\theta^nd\theta)= V(f(X)\theta^n)d\theta -f(X)\theta^nd(x\theta^2) = (\ldots)d\theta - f(X)\theta^{n+2}dx
\]
and therefore we get together with the above all 1-forms of the form $f(X)\theta^n dx$ where $n$ is even and $f$ invariant and similarly the ones with $n$ odd
and $f$ anti-invariant. In the identical way we get all 1-forms $f(X)\theta^n dy$ by taking the invariant vector field $W=y\theta(x\partial/\partial x
-y\partial/\partial y + \theta\partial/\partial\theta)$ instead. The invariant vector fields $xz\theta\partial/\partial\theta$,
$yz\theta\partial/\partial\theta$ and $z^2\theta\partial/\partial\theta$ will help to construct all forms of the form $f(X)\theta^ndz$. Indeed the calculations
\[
 L_{xz\theta\frac{\partial}{\partial \theta}}(f(X)\theta^nd\theta)=(\ldots)d\theta + (\ldots)dx+ xf(X)\theta^{n+1}dz,
\]
\[
 L_{yz\theta\frac{\partial}{\partial \theta}}(f(X)\theta^nd\theta)=(\ldots)d\theta + (\ldots)dy+ yf(X)\theta^{n+1}dz,
\]
\[
 L_{z^2\theta\frac{\partial}{\partial \theta}}(f(X)\theta^nd\theta)=(\ldots)d\theta + 2zf(X)\theta^{n+1}dz,
\]
show that we get all 1-forms of the form $g(X)\theta^ndz$ where $g(X)$ is either a multiple of $x$, $y$ or $z$. Hence allowing linear combinations only
the constant term $\theta^ndz$ is missing. But since the form $\theta^n dz + n\theta^{n-1}zd\theta$ is closed the corresponding vector field also corresponds to
$-n\theta^{n-1}zd\theta$ and hence is already obtained. 
\end{proof}

The question  we like to investigate in the remaining part of the paper is, whether the hyperbolic vector fields are needed  in the proof of Theorem
\ref{satzchar} or if the shear fields could be enough. In the following section it is shown that the Lie algebra generated by the shear fields doesn't contain
all the hyperbolic fields. Here are some preliminaries.
The proof  of the first fact is an easy consequence of the Jacobi identity.
\begin{lemma} \label{lemschachtelung}
 Let $M$ be a set of vector fields, then the Lie algebra $\mathrm{Lie}(M)$ generated by $M$ is spanned (as a vector space) by elements of the form \linebreak
$[A_n,[..[A_2,[A_1,A_0]]..]]$ with $ A_i\in M$.
\end{lemma}

In order to study which polynomials correspond to Lie combinations of shear fields it is therefore necessary to study functions of the type
$i_{[A_n,[..[A_2,[A_1,A_0]]..]]}\omega$ where $A_i$ are shear fields.
\begin{lemma}  \label{lempolynom}
 Let $A_i$ be shear fields for $0\leqslant i\leqslant n$, then the polynomial $f$ with $i_{[A_n,[..[A_2,[A_1,A_0]]..]]}\omega = df$ is of type (a) $x^j q(z)$,
(b) $y^j q(z)$ or (c) $q(z)$ for some $j>0$ and some polynomial $q\in\C[z]$.
\end{lemma}
\begin{proof}
 For $n=0$ the claim holds due theorem $\ref{satzi}$. \newline \underline{(a)} If $f=x^j q(z)$, lemma $\ref{lemi}$ shows
\begin{equation*}                                                                                                i_{[SF_k^x,[A_n[..[A_2,[A_1,A_0]]..]]}\omega=
dL_{SF_k^x}(x^j q(z))= dx^{j+k+1}q'(z),
\end{equation*}
hence the polynomial is again of the type (a). Furthermore:
\begin{eqnarray*}
i_{[SF_k^y,[A_n[..[A_2,[A_1,A_0]]..]]}\omega&=& dL_{SF_k^y}(x^j q(z))\\&=& dp'(z)y^k jx^{j-1}q(z) + y^{k+1}x^jq'(z)\\
 &=& dy^k x^{j-1}(jp'(z)q(z)+xyq'(z))\\ &=& dy^k x^{j-1}(jp'(z)q(z)+p(z)q'(z)).                                   \end{eqnarray*}
After substituting $xy=p(z)$ this polynomial is also of the type (a),(b) or (c), depending weather $k<j-1$, $k>j-1$  or $k=j-1$. Similarly it holds
that:\newline \underline{(b)} If $f=y^j q(z)$ then
\begin{eqnarray*}
 i_{[SF_k^x,[A_n[..[A_2,[A_1,A_0]]..]]}\omega&=& dx^k y^{j-1}(jp'(z)q(z)+p(z)q'(z)),\\
i_{[SF_k^y,[A_n[..[A_2,[A_1,A_0]]..]]}\omega&=& dy^{j+k+1}q'(z).
\end{eqnarray*}
\underline{(c)} If $f=q(z)$ then
\begin{eqnarray*}
 i_{[SF_k^x,[A_n[..[A_2,[A_1,A_0]]..]]}\omega&=&dx^{k+1}q'(z),\\
i_{[SF_k^y,[A_n[..[A_2,[A_1,A_0]]..]]}\omega&=& dy^{k+1}q'(z).
\end{eqnarray*}
\end{proof}
\begin{lemma}\label{f}
 If the case (c) in lemma $\ref{lempolynom}$ occurs, that is $i_{[A_n,[..[A_2,[A_1,A_0]]..]]}\omega = df$ for $A_i$ shear fields, and in addition $f=f(z)$ for
some polynomial in $z$, then $f(z)=(p(z)q(z))'$ for some polynomial $q$ in $z$.
\end{lemma}
\begin{proof}
 Consider the vector field $[A_{n-1},[..[A_2,[A_1,A_0]..]]$. Due to lemma $\ref{lempolynom}$ exactly one of the following cases occurs:
\[i_{[A_{n-1},[..[A_2,[A_1,A_0]]..]]}\omega = \begin{cases}
dx^j q(z)&(a)\\dy^j q(z)&(b)\\dq(z)&(c)
\end{cases}
 \]for some $j>0$ and some $q\in \C[z]$.
If $A_n = SF_k^x$ for some $k\in\N_0$, then together with the calculation in the proof of lemma $\ref{lempolynom}$ one gets:
 \[i_{[A_{n},[A_{n-1}[..[A_2,[A_1,A_0]]..]]]}\omega = df=\begin{cases}
dx^{j+k+1}q'(z)&(a)\\dx^k y^{j-1}(jp'(z)q(z)+p(z)q'(z))&(b)\\dx^{k+1}q'(z)&(c).
\end{cases}
 \]
Since $f$ is a polynomial in $z$ all cases except (b) with $k=j-1$ can be excluded. Therefore $f = x^k y^{j-1}(jp'(z)q(z)+p(z)q'(z)) = p(z)^k
((k+1)p'(z)q(z)+p(z)q'(z)) = (p(z)^{k+1}q(z))'$ for some $q\in \C[z]$.
The identical consideration works for $A_n = SF_k^y$.
\end{proof}
\begin{remark} \label{bemoab}
 If one chose in the last step $k=j+i-1$ instead of $k=j-1$ for some $i\in\N$, the polynomial in the end of the calculation would have been $x^i(pq)'=x^i f(z)$
(respectively $y^i f(z)$). Hence if $f(z)$ corresponds to a Lie combination of shear fields, then so does the polynomial $x^i f(z)$ (respectively $y^i f(z)$).
By permuting $SF_i^x$ and $SF_i^y$ the corresponding polynomial switches the sign and $x$ and $y$ get permuted, hence both $x^i f(z)$ and $y^i f(z)$ correspond
to Lie combinations of shear fields.
\end{remark}

\begin{corollary}
If a hyperbolic vector field is a Lie combination of shear fields then it is of the form $HF_{(pq)''}$ for some $q\in\C[z]$. In particular if $p\in\C[z]$
with degree $n \geqslant 3$, then the the hyperbolic vector fields $HF_{z^i}$ with $i< n-2$ are not Lie combinations of shear fields.
\end{corollary}

In addition we can make the following observation: 

\begin{corollary} \label{dense} For $p\in\C[z]$ with degree $n \geqslant 3$ the Lie algebra generated by holomorphic shear fields is not dense in the Lie
algebra of holomorphic volume preserving vector fields.
\end{corollary}

\begin{proof}  Formula $(1)$ for a regular function$ f$ on $D_p$ can be viewed as a Laurent expansion of the restriction of $f$ to the open subset $x\ne 0 \cong
\C^*_x \times \C_z$
with respect to the variable $x \in \C^*$ with coefficients being  functions of $z$. Analogously any holomorphic function $g$ on $D_p$ has such a Laurent
expansion with
coefficients holomorphic functions in $z$

$$ g = \sum_{i = - \infty}^\infty a_i (z) x^i .$$

We have established that the regular function $f$ corresponding under  $i_{\Theta}\omega = df$ to an  algebraic vector field $\Theta$ which is a Lie combination
of algebraic shear fields  satisfy the special condition $a_0 (z)  = (h p)^\prime$, i.e., the absolute term $a_0 (z)$ (which is unique associated to $\Theta$ up
to a constant) is the derivative of a function
divisible by the defining polynomial $p$. The condition that a function $g$ on $\C^* \times \C$ has an absolute term which is up to a constant  the derivative
of a function
divisible by the defining polynomial $p$ is closed in c.-o. topology. More explicitly, let $z_1, \ldots , z_n$ be the distinct simple zeros of $p$, then the
condition is
equivalent to the equality of  all the expressions

$$ (z_j -z_1) \int_{z_1}^{z_j} \int_{\vert x\vert = 1}  g (x,z) \frac{ dz \wedge dx}{x}  \quad j= 2, 3, \ldots , n .$$ Since holomorphic shear fields are limits
(in c.-o. topology) of
algebraic shear fields the holomorphic function corresponding to a Lie combination of holomorphic shear fields has an absolute term of the same form. 
Thus for $p$ with degree $\ge 3$ the Lie algebra generated by holomorphic shear fields is contained in the closed proper subset of the Lie algebra of
holomorphic volume preserving vector fields defined by the above condition on the absolute term.
\end{proof}

%%%%%%%%%%%%%%%%%%%%%%%%%%%%%%%%%%%%%%%%%%%%%%%%%%%%%%%%%%%%%%%%%%%%%%%%%%%%%%%%%%

\subsection{Description of the Lie Algebra generated by Shear Fields}
After negating the question whether every volume preserving vector field is a Lie combination of shear fields, in the this section it will be investigated
which vector fields exactly are Lie combination of such ones. Concretely all of the volume preserving vector fields whose absolute term of the corresponding
function is
of the special 
form described in Lemma \ref{f}   are a Lie combination of shear fields. This proof is following the same concept developed in  \cite{lind} to prove  the fact
that the shear fields and another class of (non volume preserving) vector fields called overshear fields do generate the Lie algebra of algebraic vector fields
of $D_p$.
\begin{lemma} \label{lemklammern}
 The following equalities hold:
\begin{eqnarray}
 i_{[SF_i^x,SF_i^y]}\omega&=&d(p^i p')\\
 i_{[SF_0^x,[SF_0^x,SF_1^y]]}\omega&=&d(pp')' \\
 i_{[SF_{i_k - 1}^x, .. [SF_{i_2 - 1}^x,[SF_{i_1 -1}^x,HF_f]]..]}\omega &=& d(x^{i_1 + .. + i_k}f^{(k-1)})\label{ableitung}\\
 i_{[HF_{f_k}, .. [HF_{f_2},[SF_i^x,HF_{f_1}]]..]}\omega&=&(i+1)^{k-1}d(x^{i+1} f_1 f_2 .. f_k). \label{lemklammernprodukte}
\end{eqnarray}

\end{lemma}
\begin{proof}
 The following calculations are according to theorem $\ref{satzi}$ and lemma $\ref{lemi}$:\newline
\begin{eqnarray*}
 (1) \quad i_{[SF_i^x,SF_i^y]}\omega &=& dL_{SF_i^x}\left(\frac{y^{i+1}}{i+1}\right)\\
&=& dp'(z)x^i y^i.
\end{eqnarray*}
\begin{eqnarray*}
(2) \quad \quad i_{[SF_0^x,SF_1^y]}\omega &=&dyp'(z)\\
i_{[SF_0^x,[SF_0^x,SF_1^y]]}\omega&=&dL_{SF_0^x}\left(yp'(z)\right)\\
&=&dp'(z)p'(z) + xyp''(z)\\
&=&d(p(z)p'(z))'.
\end{eqnarray*}
\begin{eqnarray*}
 (3) \quad \quad i_{[SF_{i_1 -1}^x,HF_f]}\omega&=&dx^{i_1}f\\
i_{[SF_{i_2 -1}^x,[SF_{i_1 -1}^x,HF_f]]}\omega&=&dL_{SF_{i_2 -1}^x}(x^{i_1}f)\\
&=&x^{i_1+i_2}f'\\
i_{[SF_{i_k - 1}^x, .. [SF_{i_2 - 1}^x,[SF_{i_1 -1}^x,HF_f]]..]}\omega &=& dL_{SF_{i_k - 1}^x}(x^{i_1+i_2 + .. + i_{k-1}}f^{(k-2)})\\
&=&d(x^{i_1+i_2 + .. + i_{k}}f^{(k-1)}).
\end{eqnarray*}
\begin{eqnarray*}
 (4) \quad \quad i_{[SF_i^y,HF_{f_1}]}\omega&=&dx^{i+1}f_1\\
i_{[HF_{f_2},[SF_i^x,HF_{f_1}]]}\omega&=&dL_{HF_{f_2}}(x^{i+1}f_1)\\
&=&(i+1)x^{i+1}f_1 f_2\\
i_{[HF_{f_k}, .. [HF_{f_2},[SF_i^x,HF_{f_1}]]..]}\omega&=&dL_{HF_{f_k}}((i+1)^{k-2}x^{i+1}f_1 .. f_{k-1})\\
&=&d((i+1)^{k-1}x^{i+1}f_1 .. f_{k}).
\end{eqnarray*}

\end{proof}
\begin{corollary} \label{corklammern}
 The previous lemma shows:
\begin{eqnarray}
 [SF_i^x,SF_i^y]&=&HF_{(p^i p')'}  \label{anfang} \\
\left[SF_0^x,[SF_0^x,SF_1^y]\right]&=&HF_{(pp')''}\\
\left[SF_{i_1 + .. + i_k - 1}^y,[SF_{i_k - 1}^x, .. [SF_{i_1 -1}^x,HF_f]..]\right]&=&HF_{(p^{i_1 + .. + i_k}f^{(k-1)})''}\\
\left[SF_i^y,[HF_{f_k}, .. [SF_i^x,HF_{f_1}]..]\right]&=&HF_{(i+1)^{k-1}(p^{i+1} f_1 f_2 .. f_k)''}.
\end{eqnarray}

\end{corollary}

\begin{lemma} \label{lempn}
 Let $n=\deg(p)$, then for every $q\in\C[z]$ the vector field $HF_{(p^{n}q)''}$ is a Lie combination of shear fields.
\end{lemma}
\begin{proof}
 In a first step one observes that every polynomial $x^{n}q$ corresponds to a Lie combination of shear fields. Truly due to lemma $\ref{lemklammern}
(\ref{ableitung})$ the polynomials $x^{n}f^{(k)}$ for $k=0$, .. ,$(n-1)$ correspond to a Lie combination of shear fields, if $HF_f$ was already such a
combination. According to corollary $\ref{corklammern} (\ref{anfang})$ it is possible to choose for $f$ the polynomials $p''$, $(pp')'$, $(p^2 p')'$, ...  (i.e.
polynomials of degree $n-2$, $2n-2$, $3n-2$, ...). Therefore after differentiating  up to $n$ times there is a polynomial for every degree and hence they build
a basis for $\C[z]$ and every polynomial $q\in\C[z]$ can be substituted in $x^{n}q$. After taking the Lie bracket with the shear field $SF_{n-1}^y$ the vector
field becomes $HF_{(p^n q)''}$.
\end{proof}
Let $n=\deg p$ and $W \subset \C[z]$ be a vector space with
\begin{eqnarray*}
 (i) &(p^i)'' \in W & \forall i\in \N\\
(ii) &(pp')'' \in W&\\
(iii)&(p^{n}q)'' \in W & \forall q\in \C[z]\\
(iv)&f_1, ..,f_k \in W \Longrightarrow (p f_1 .. f_k)'' \in W & \forall k\in \N.
\end{eqnarray*}
Now the goal is to show that $W$ contains all polynomials of the type $(pq)''$. Since the vector space of all $f$ with $HF_f$ a Lie combination of shear fields
is a vector space with properties (i)-(iv), every vector field $HF_{(pq)''}$ would be a such combination.\newline

In a first step it is shown that the algebra $A_W = \mathrm{span}\lbrace f_1\cdot\ldots\cdot f_k: \ f_i \in W, \ 1\leq i \leq k \in \N \rbrace$ generated by
$W$ is equal to $\C[z]$. Then it is allowed to substitute all polynomials in (iv) and hence the
claim is proven.

\begin{lemma}\label{lemcomroot}
There is no element $a\in \C$, such that $f(a)=0$ for all $f\in A_W$.
\end{lemma}

\begin{proof}
 Suppose there is such an $a$, then $p''(a)=0$ and $p(a)p''(a)+p'(a)^2=0$ ((i) with $i=1$ and $i=2$) would hold, and hence $p'(a)=0$. Since $p$ has no double
zero point it follows that $p(a)\neq 0$. Due to (iii) $(p^n q)''(a)= (p^n)''q+2(p^n)'q' +p^n q''(a) = 0$ holds for all $q\in \C[z]$. The first summand vanishes
due to (i), the second due to $p'(a)=0$, therefore it remains $p^n q''(a)=0$. So it would be true that $q''(a)=0$ for all $q\in\C[z]$ what is clearly a
contradiction.

\end{proof}

\begin{lemma} \label{lemimm}
There is no element $a\in \C$ such that $f'(a)=0$ for all $f\in A_W$.
\end{lemma}

\begin{proof}
 Suppose there is such an $a$. (i) with $i=1,2,3$ shows that $p'''(a) =0$, $(p^2)'''(a)=2(p(a)p'''(a)+3p'(a)p''(a))=0$ and $0=(p^3)'''(a)=3(p(a)^2 p'''(a)
+6p(a)p'(a)p''(a)+2p'(a)^3)$. The second equation shows that $p'(a)p''(a)=0$ and therefore due to the third equation we have $p'(a)=0$ and hence $p(a)\neq0$.
Furthermore (iii) shows $(p^n z)'''(a) = (p^n)'''(z)z + 3(p^n)''(z)\mid_{z=a}=0$ and since the first summand vanishes $(p^n)''(a)=0$ remains. Altogether we
have $(p^n q)'''(a)=((p^n)'''q+3(p^n)''q'+3(p^n)'q'' +p^n q''')(a)=p(a)^n q'''(a)=0$ or $q'''(a)=0$ for all $q\in\C[z]$, what is again a contradiction.

\end{proof}

\begin{lemma} \label{leminj}
There are no elements $a\neq b \in \C$, such that $f(a)=f(b)$ for all $f\in A_W$.
\end{lemma}
\begin{proof}
 Suppose there are two such elements $a,b \in \C$. (iv) shows that \linebreak $(p^n z^i)''\mid_{z=a} = (p^n z^i)''\mid_{z=b}$ for all $i\in \N_0$. Since $(p^n
z^i)''=(p^n)''z^i + 2i(p^n)'z^{i-1} +i(i-1)p^n z^{i-2}$ one gets the system of linear equations, which summarizes the equations for $i=0,\ldots,5$:

\[   \left( \begin {array}{cccccc} 1&1&0&0&0&0\\\noalign{\medskip}a&b&1&1&0
&0\\\noalign{\medskip}{a}^{2}&{b}^{2}&2\,a&2\,b&2&2
\\\noalign{\medskip}{a}^{3}&{b}^{3}&3\,{a}^{2}&3\,{b}^{2}&6\,a&6\,b
\\\noalign{\medskip}{a}^{4}&{b}^{4}&4\,{a}^{3}&4\,{b}^{3}&12\,{a}^{2}&
12\,{b}^{2}\\\noalign{\medskip}{a}^{5}&{b}^{5}&5\,{a}^{4}&5\,{b}^{4}&
20\,{a}^{3}&20\,{b}^{3}\end {array} \right) \cdot  \left( \begin {array}{c} (p^n)''(a)\\\noalign{\medskip}-(p^n)''(b)\\\noalign{\medskip}2(p^n)'(a)
\\\noalign{\medskip}-2(p^n)'(b)\\\noalign{\medskip}(p^n)(a)\\\noalign{\medskip}-(p^n)(b)
\end {array} \right)  =  \left( \begin {array}{c} 0\\\noalign{\medskip}0\\\noalign{\medskip}0
\\\noalign{\medskip}0\\\noalign{\medskip}0
\\\noalign{\medskip}0\end {array} \right) 
\]
The determinant of this matrix is $4(a-b)^9$ and therefore nonzero for $a\neq b$ and hence it is shown that the coefficient vector is the zero vector and in
particular $p(a)=p(b)=0$ and therefore $p'(a)\neq 0 \neq p'(b)$. \newline
Due to (ii) we have $(pp')''(a) = p(a)p'''(a)+3p'(a)p''(a) = p(b)p'''(b) + p'(b)p''(b)$ and since $p(a)=p(b)=0$ and $p''(a)=p''(b)$ (due to (i)) $p'(a)=p'(b)$
holds. With (iv) ($k=1$) follows $(pp'')''(a)=p(a)p''''(a) + 2p'(b)p'''(b)+p''(b)^2=p(b)p''''(b) + 2p'(b)p'''(a)+p''(b)^2$ and hence $p'''(a) = p'''(b)$.
Using (iv) inductively one gets that $p^{(l)}(a)=p^{(l)}(b)$ for all $l$. Indeed a simple calculation shows that:
\[W \ni P:= \left(\underbrace{p(p(p \cdots (p}_{j}p'')''\ldots\ )'' )''\right)'' = \sum_{\substack{i_1+ \ldots +i_{j+1}=2j+2\\i_1\leq .. \leq i_{j+1}}}
\alpha_I\cdot p^{(i_1)}
\cdots p^{(i_{j+1})}
\]
with $a_I \in \N$. After inserting $a$ (resp. $b$) all summands with $i_1 =0$ vanish due to $p(a)=p(b)=0$. Assume that $p^{(l)}(a) =
p^{(l)}(b)$ for all $l\leq j+1$, so all the summands with $i_{j+1}\leq j+1$ have on both sides of the equation $P(a)=P(b)$ the same value and hence vanish as
well. For this reason only the equation $\alpha_I p'(a)^{j}p^{(j+2)}(a) =\alpha_I p'(b)^{j}p^{(j+2)}(b)$ remains and it follows inductively that $p^{(l)}(a) =
p^{(l)}(b)$ for all $l$.  This is a contradiction since the $(n-1)$-st derivative of a polynomial of degree $n$ is a polynomial of degree one with a nonzero
slope.
\end{proof}

\begin{proposition} \label{propAW}
 The algebra $A_W$ generated by $W$ is equal to $\C[z]$.
\end{proposition}
\begin{proof}
 The previous two lemmas show that there is a $k\in\N$ and polynomials $q_1$, ... ,$q_k \in A_W$ such that the map
\[F: \C \rightarrow \C^k: \qquad z \mapsto (q_1(z), \ldots, q_k(z))\]
is an injective and immersive embedding. To achieve injectivity take the ideal in $\C[x,y]$ generated by the polynomials $q(x)-q(y)$ with $q\in A_w$ which is
finitely generated by polynomials $q_1(x)-q_1(y)$, .. , $q_k(x)-q_k(y)$. Now we see that there are no $c_1 \neq c_2$ such that $q_i(c_1)=q_i(c_2)$ for all $i$,
otherwise we would have $q(c_1)=q(c_2)$ for all $q \in A_W$ which is not possible due to lemma $\ref{leminj}$. To guarantee immersivity we add for each cusp
singularity (finite number!) a polynomial $q \in A_W$ whose derivative doesn't vanish at this point (lemma $\ref{lemimm}$). \newline

Now take any polynomial function $g$ on $\C$ and regard it as a regular function on the by $F$ embedded $\C$ in $\C^k$. This function expands to a regular
function $G$ on $\C^n$ hence $G = a_0 + \sum_I a_I z_1^{i_1}\cdots z_k^{i_k}$. So if we pull back $G$ we get $g(z)=G(F(z))=a_0 + \sum_I a_I q_1(z)^{i_1}\cdots
q_k(z)^{i_k}$ so the algebra generated by $q_1, \ldots, q_k$ and constants is $\C[z]$. Now the algebra generated by $W$ is $\C[z]$ or a subspace with
codimension 1 and an ideal hence in the second case $W$ is a principle ideal generated by a polynomial $(z-a)$. But this case can't occur since $a$ would be a
common root of all elements of $W$ what is impossible (lemma $\ref{lemcomroot}$).
\end{proof}

Now we know that a hyperbolic field $HF_f$ is a Lie combination of shear fields if and only if $f=(pq)''$ for some polynomial $q$. In theorem $\ref{satzchar}$
it was shown that every volume preserving vector field is a linear combination of the vector fields $SF_i^x$, $SF_i^y$, $HF_f$, $[SF_i^x,HF_f]$ and
$[SF_i^y,HF_f]$ for $i\in\N_0$ and polynomials $f\in\C[z]$. To understand which vector fields are Lie combinations of shear fields it remains to study
the vector fields $[SF_i^x,HF_f]$ and $[SF_i^y,HF_f]$. 

\begin{proposition} \label{propmixedterms}
 All the vector fields $[SF_i^x,HF_f]$ and $[SF_i^y,HF_f]$ for $i\in\N_0$ and polynomials $f\in\C[z]$ are Lie combinations of shear fields.
\end{proposition}
\begin{proof}
 Since $i_{[SF_{i-1}^x,HF_{f}]}\omega =dx^i f$ it suffices to see that the polynomial $x^i f(z)$ corresponds for every $i\in\N$ and every $f\in\C[z]$ to a Lie
combination of shear fields. In the proof of lemma $\ref{lempn}$ we saw that this is true for $i\geq n =\deg p$. So we already have every $x^n z^j$ for
$j\in\N$. If one takes the Lie bracket with the vector field $SF_0^y$ one gets with the calculation in the proof of lemma $\ref{lempolynom}$ (a) the polynomial
$x^{i}((i+1)p'(z)z^j + jp(z)z^{j-1})$ for $i=n-1$. Since every polynomial $(p(z)z^j)'$ corresponds to a Lie combination of shear fields, so does the polynomial
$x^{i}(p(z)z^j)'=x^{i}(p'(z)z^j + jp(z)z^{j-1})$ (due to remark $\ref{bemoab}$). After a suitable linear combination of this two polynomials it follows that
$x^{i}p'(z)z^j$ and $x^{i}p(z)z^{j-1}$ correspond to a Lie combination of shear fields for all $j$. Therefore every $x^{i}f(z)$ with $f(z) \in (p) \cup (p')
\subset \C[z]$ belongs to a Lie combination. Since $p$ and $p'$ have no common zeros it is true that $(p) \cup (p') = \C[z]$ and the claim is shown for $i=n-1$.
Repeat the same procedure for $i=n-2, \ldots, 1$ and the claim is shown for every $i \in \N$.
\end{proof}

Now we have to make the final step allowing not only shear fields but also LNDs in our Lie combination. Since LNDs are shears conjugated by compositions of
shear automorphisms
(see theorem $\ref{shearlnd}$) the following lemma will do the job. 
\begin{lemma} \label{lemconjugateshear}
 Let $\phi: D_p \rightarrow D_p$ be a shear automorphism and let $\Theta$ be a Lie combination of shear fields. Then $\phi^*\Theta$ is a Lie combination of
shear fields.
\end{lemma}
The proof of this lemma follows immediately from the following general fact.

\begin{lemma} \label{lemclosed}
Let $\Theta$ be an  LND  with flow $\phi_t$ and $\Psi$ any algebraic vector field.  Then for any fixed $t$ the vector field $(\phi_t)^* (\Psi)$ is contained in
the Lie algebra generated by
$\Theta$ and $\Psi$.
\end{lemma}

\begin{proof}

Since $\Theta$ is an LND the Taylor expansion of $(\phi_t)^* (\Psi)$ with respect to the variable $t$ around $t_0 = 0$ 

$$(\phi_t)^* (\Psi) = \Psi + t [\Theta, \Psi] + \frac{1}{2} t^2 [\Theta, [\Theta, \Psi]] + \ldots + \frac{1}{n !} t^n [\Theta,[\Theta \ldots [\Theta, \Psi]]
\ldots]$$
is a polynomial in $t$. This implies the claim.
\end{proof}

Thus we can now proof the main result.

\begin{theorem} \label{final}
 A volume preserving vector field $\Theta$ on the Danielewski surface $D_p$ is a Lie combination of LNDs if and only if its corresponding function with
$i_\Theta\omega = df$ is of the form (modulo constant)
\[ f(x,y,z) = \sum_{\substack{i=1\\j=0}} ^{k}a_{ij} x^i z^j + \sum_{\substack{i=1\\j=0}}^{l}b_{ij} y^i z^j + (pq)'(z)
\] 
for a polynomial $q\in\C[z]$.
\end{theorem}

\begin{proof}
 By Proposition \ref{propAW} together with Lemma \ref{lemklammern} (\ref{lemklammernprodukte}) and Proposition \ref{propmixedterms} the Lie algebra generated by
shear fields consists exactly of those volume preserving fields described in the theorem. By Theorem \ref{shearlnd} any LND $\Theta$ is conjugated to a shear
field $S$ by an automorphism $\psi$ which is a finite composition of shear automorphisms $\psi = \alpha_m \circ \ldots \circ \alpha_1$. Thus by Lemma
\ref{lemconjugateshear}  $\Theta=\psi^*S=\alpha_1^*(\ldots \alpha_{m-1}^*(\alpha_m^*S)\ldots)$ is contained in the Lie algebra generated by shear fields.
\end{proof}

%\bibliographystyle{annaliscienze}
%\bibliography{bibtex}

\end{document}